\documentclass[a4paper,12pt,twoside]{amsart}
\usepackage[english]{babel}
\usepackage{amssymb, amsmath, eucal}
\usepackage[all]{xy}
\usepackage{mathrsfs}
\usepackage[dvips]{graphics}

\newtheorem{The}{Theorem}
\newtheorem{Lem}[The]{Lemma}

\newtheorem{Cor}[The]{Corollary}

\newtheorem{Question}[The]{Question}

\newtheorem*{ackn}{Acknowledgements}
\newcommand{\B}{\mathbb{B}}

\newcommand{\C}{\mathbb{C}}

\newcommand{\N}{\mathbb{N}}
\newcommand{\Z}{\mathbb{Z}}

\newcommand{\E}{\mathcal{E}}
\newcommand{\F}{\mathcal{F}}

\begin{document}
 \title[ Cegrell's class $\mathcal{F}$]{Some remarks on the Cegrell's class $\mathcal{F}$}

\setcounter{tocdepth}{1}

  \author{Hoang-Son Do} 
\address{Institute of Mathematics \\ Vietnam Academy of Science and Technology \\18
Hoang Quoc Viet \\Hanoi \\Vietnam}
\email{hoangson.do.vn@gmail.com , dhson@math.ac.vn}
\author{Thai Duong Do}
\address{Institute of Mathematics \\ Vietnam Academy of Science and Technology \\18
	Hoang Quoc Viet \\Hanoi \\Vietnam}
\email{dtduong@math.ac.vn}
 \date{\today\\ This research is funded by Vietnam National Foundation for Science and Technology Development (NAFOSTED) under grant number 101.02-2017.306.}


\begin{abstract}
  In this paper, we study the near-boundary behavior of functions $u\in\F (\Omega)$ in the case where $\Omega$ is strictly pseudoconvex.
  We also introduce a  sufficient condition for belonging to $\F$ in the case where $\Omega$ is the unit ball.
\end{abstract}

\maketitle

\section*{Introduction}
Let $\Omega$ be a bounded hyperconvex domain in $\C^n$. By \cite{Ceg04}, the class $\F (\Omega)$ is defined
as the following: $u\in\F (\Omega)$ iff there exists a sequence of functions
 $u_j\in \E_0(\Omega)$ such that $u_j\searrow u$ as $j\rightarrow\infty$ and 
 $\sup_j\int_{\Omega}(dd^cu_j)^n<\infty$. Here
 \begin{center}
 	$\E_0(\Omega)=\{u\in PSH(\Omega)\cap L^{\infty}(\Omega):
 	\lim\limits_{z\to\partial\Omega}u(z)=0, \int\limits_{\Omega}(dd^cu)^n<\infty \}.$
 \end{center}
 
 The class $\F(\Omega)$ has many nice properties. This is a subclass of the domain of  definition of Monge-Amp\`ere operator 
 \cite{Ceg04, Blo06}. Moreover, by \cite{Ceg04}, for each sequence of functions
 $u_j\in \E_0(\Omega)$ such that $u_j\searrow u\in\F (\Omega)$ as $j\rightarrow\infty$, we have
 \begin{center}
 	$\lim\limits_{j\to\infty}\int\limits_{\Omega}(dd^cu_j)^n=\int\limits_{\Omega}(dd^cu)^n$.
 \end{center}
 .
 
 By \cite{Ceg98, Ceg04},  for every pluripolar
 set $E\subset\Omega$, there exists $u\in\F (\Omega)$ such that $E\subset\{u=-\infty\}$. In \cite{Ceg04},
 Cegrell also proved some inequalities, a 
 generalized comparison principle and a decomposition of $(dd^cu)^n, u\in\F (\Omega)$. 
 In \cite{NP09}, Nguyen and Pham proved a strong version of comparison principle in the class $\F(\Omega)$.
 
 The class $\F (\Omega)$ has been used to characterize the boundary behavior in the
 Dirichlet problem for Monge-Amp\`ere equation \cite{Ceg04, Aha07}. For every $u\in\F (\Omega)$,
 for each $z\in\partial\Omega$,
 we have 	$\limsup\limits_{\Omega\ni\xi\to z}u(\xi)=0$ (see \cite{Aha07}). Moreover, if we define by $\mathcal{N}$ the 
 set of functions
 in the domain of  definition of Monge-Amp\`ere operator with smallest maximal plurisubharmonic majorant identically zero then,
 by the comparison principles in $\F$ and in $\mathcal{N}$  (see \cite{NP09} and \cite{ACCP09})  and by Cegrell's approximation theorem
 \cite{Ceg04} (see also Lemma \ref{the cegrell appro}), we have
 \begin{center}
 	$\F (\Omega)=\{u\in\mathcal{N}(\Omega): \int\limits_{\Omega}(dd^cu)^n<\infty \}$.
 \end{center}
 
 In this paper, we study the near-boundary behavior of functions $u\in\F(\Omega)$ in the case where 
 $\Omega$ is a bounded strictly pseudoconvex domain, i.e., there exists $\rho\in PSH(\Omega)\cap C(\overline{\Omega})$ such that
   $\rho|_{\partial\Omega}=0$, $D\rho|_{\partial\Omega}\neq 0$ and $dd^c\rho\geq c\omega:=cdd^c|z|^2$ for some $c>0$.
   
   Our first main result is the following:
 \begin{The}\label{main 1}
 	Assume that $\Omega$ is a strictly pseudoconvex domain in $\C^n$ and $u\in \F(\Omega)$. 
 	Then, there exists $C>0$ depending only on  $\Omega, n$ and $u$ such that \\
 	\begin{equation}
 		Vol_{2n}(\{z\in\Omega| d(z, \partial\Omega)<d, u(z)<-\epsilon \})\leq\dfrac{C.d^{n+1-na}}{a^n\epsilon^n},
 	\end{equation}
 	for any $\epsilon, d>0$, $a\in (0, 1)$.
 \end{The}
 For the convenience, we denote $W_d=\{z\in\Omega| d(z, \partial\Omega)<d \}$. 
 By Theorem \ref{main 1}, we have
 \begin{center}
 	$\lim\limits_{d\to 0}\dfrac{Vol_{2n}(\{z\in W_d|  u(z)<-\epsilon \})}{d^t}=0$,
 \end{center}
 for every $0<t<n+1$. It helps us to estimate the ``density'' of the the set $\{u<-\epsilon\}$ near the
 boundary.
 
  Moreover,  by using Theorem \ref{main 1} for $\epsilon=d^{\alpha}$
 and $0<a<1-\alpha$, we have
 \begin{Cor}
 		Assume that $\Omega$ is a strictly pseudoconvex domain in $\C^n$ and $u\in \F(\Omega)$. 
 	Then, for every $0<\alpha<1$,
 	\begin{center}
 		$\lim\limits_{d\to 0}\dfrac{Vol_{2n}(\{z\in W_d|  u(z)<-d^{\alpha} \})}{d}=0$.
 	\end{center}
 \end{Cor}
When $\Omega$ is the unit ball, this result can be improved as following:
\begin{The}\label{main 2}
	If $u\in\F(\B^{2n})$ then
	\begin{center}
		$\lim\limits_{r\to 1^-}\dfrac{\int_{\{|z|=r\}}|u(z)|d\sigma(z)}{1-r}<\infty$.
	\end{center}
In particular, there exists $C>0$ such that
\begin{center}
		$\limsup\limits_{d\to 0^+}\dfrac{Vol_{2n}(\{z\in\mathbb{B}^{2n} : \|z\|>1-d, u(z)<-Ad\})}{d}<\dfrac{C}{A},$
\end{center}
for every $A>0$.
\end{The}
 Our second purpose is to find a sharp sufficient condition for $u$ to belong to $\F(\Omega)$ based on the 
 near-boundary behavior of $u$. We are interested in the following question:
 \begin{Question}
 	Let $\Omega$ be a bounded strictly pseudoconvex domain. Assume that $u$ is a negative plurisubharmonic function in $\Omega$
 	satisfying
 	\begin{center}
 		$\lim\limits_{d\to 0^+}\dfrac{Vol_{2n}(\{z\in W_d : u(z)<-Ad\})}{d}=0,$
 	\end{center}
 	for some $A>0$. Then, do we have $u\in\F(\Omega)$?
 \end{Question}
 In this paper, we answer this question for the case where $\Omega$ is the unit ball.
 \begin{The}\label{main 3}
 	Let $u\in PSH^-(\mathbb{B}^{2n})$. Assume that there exists $A>0$ such that
 	\begin{equation}\label{eq main3}
 		\lim\limits_{d\to 0^+}\dfrac{Vol_{2n}(\{z\in\mathbb{B}^{2n} : \|z\|>1-d, u(z)<-Ad\})}{d}=0.
 	\end{equation}
 	Then $u\in \F(\mathbb{B}^{2n})$.
 \end{The}
\begin{Cor}
	Let $u\in\mathcal{N}(\mathbb{B}^{2n})$ such that $\int\limits_{\mathbb{B}^{2n}}(dd^cu)^n=\infty$. Then, for every $A>0$,
	\begin{center}
		$\limsup\limits_{d\to 0^+}\dfrac{Vol_{2n}(\{z\in\mathbb{B}^{2n} : \|z\|>1-d, u(z)<-Ad\})}{d}>0.$
	\end{center}
\end{Cor}
\begin{ackn}
	The authors would like to thank Professor Pham Hoang Hiep for valuable comments that helped them to improve the manuscript. 
\end{ackn}

\section{Proof of Theorem \ref{main 1}}
Since $\Omega$ is bounded strictly pseudoconvex, there exists $\rho\in C^2(\bar{\Omega}, [0, 1])$ such that $\Omega=\{z: \rho(z)<0\}$
and
\begin{equation}\label{drho.eq}
|D\rho|>C_1 \ \mbox{in}\ \bar{\Omega},
\end{equation}
and
\begin{equation}\label{d2rho.eq}
dd^c\rho\geq C_2 dd^c|z|^2=C_2\omega,
\end{equation}
where $C_1, C_2>0$ are constants.

 By \eqref{drho.eq}, there exist $C_3, C_4>0$ depending only on $\Omega$ and $\rho$ such that
\begin{equation}\label{eq3 proofmain1}
C_3d(z, \partial\Omega)\leq -\rho (z)\leq C_4d(z, \partial\Omega),
\end{equation}
for every $z\in\Omega$.

 For every $a\in(0, 1)$ and $z\in\Omega$, we have
\begin{center}
	$dd^c\rho_a(z):=dd^c(-(-\rho(z))^a)=a(1-a)(-\rho)^{a-2}d\rho\wedge d^c\rho+a(-\rho)^{a-1}dd^c\rho.$
\end{center}
Then
\begin{equation}\label{ddcrhoa.eq}
(dd^c\rho_a)^n\geq a^n(1-a)(-\rho)^{na-n-1}d\rho\wedge d^c\rho\wedge (dd^c\rho)^{n-1}.
\end{equation}
Hence, by \eqref{drho.eq}, \eqref{d2rho.eq} and \eqref{eq3 proofmain1}, there exists
$1\gg d_0>0$ depending only on $\Omega$ and $\rho$ such that, for every $0<d<d_0$
and $z\in  W_d:=\{\xi\in\Omega: d(\xi ,\partial\Omega)<d\}$,
\begin{equation}\label{eq4 proofmain1}
(dd^c\rho_a)^n\geq C_5d^{na-n-1}\omega^n.
\end{equation}

Since $u\in\F(\Omega)$, there exists $\{u_j\}_{j=1}^{\infty}\subset\E_0(\Omega)$ such that $u_j\searrow u$ and
\begin{equation}\label{uinF.eq}
\int\limits_{\Omega}(dd^cu_j)^n<C_6,
\end{equation}
for every $j\in\Z^+$, where $C_6>0$ depends only on $u$.

By using \eqref{eq4 proofmain1}, \eqref{uinF.eq} and the Bedford-Taylor comparison principle 
\cite{BT76, BT82} (see also \cite{Kli91}), 
we have, for every $j\in\Z^+$, $\epsilon, d>0$ and $a\in (0, 1)$,
\begin{flushleft}
	$\begin{array}{ll}
	C_6>\int\limits_{\{u_j<\epsilon\rho_a\}} (dd^c u_j)^n
	&\geq \int\limits_{\{u_j<\epsilon\rho_a\}} (dd^c \epsilon\rho_a)^n\\
	&\geq \dfrac{C_5 a^n\epsilon^n}{d^{n+1-na}}\int\limits_{\{u_j<\epsilon\rho_a\}\cap W_d}\omega^n.
	\end{array}$
\end{flushleft}
Hence, for every $0<d<d_0$,
\begin{center}
	$Vol_{2n}(\{z\in W_d| u_j(z)<-\epsilon\})\leq\dfrac{C_7.d^{n+1-na}}{a^n\epsilon^n},$
\end{center}
where $C_7>0$ depends only on $\Omega, \rho, n$ and $u$.

Letting $j\rightarrow\infty$, we get
\begin{center}
	$Vol_{2n}(\{z\in W_d| u(z)<-\epsilon\})\leq\dfrac{C_7.d^{n+1-na}}{a^n\epsilon^n},$
\end{center}
 for every $0<d<d_0$.
 
 Denote
 \begin{center}
 	$C=\max\{C_7, \dfrac{a^n\epsilon^n Vol_{2n}(\Omega)}{d_0^{n+1-na}}\}.$
 \end{center}
We have
\begin{center}
	$Vol_{2n}(\{z\in W_d| u(z)<-\epsilon\})\leq\dfrac{C.d^{n+1-na}}{a^n\epsilon^n},$
\end{center}
for every $d>0$.

This completes the proof of Theorem \ref{main 1}.
\section{Proof of Theorem \ref{main 2}}
In order to prove  Theorem \ref{main 2}, we need the following lemma:
\begin{Lem}\label{lem 1 main2} Let $\Omega\subset\C^n$ be a bounded hyperconvex domain and $(X, d, \mu)$ be a compact metric probability space. 
	Let $u: \Omega\times X\rightarrow [-\infty, 0)$ such that
	\begin{itemize}
		\item [(i)] For every $a\in X$, $u(., a)\in\F (\Omega)$ and
		\begin{center}
			$\int\limits_{\Omega}(dd^cu(z, a))^n<M$,
		\end{center}
	where $M>0$ is a constant.
	\item[(ii)] For every $z\in\Omega$, the function $u(z, .)$ is upper semicontinuous in $X$.
	\end{itemize}
Then $\tilde{u}(z)=\int\limits_Xu(z, a)d\mu(a)\in\F (\Omega)$.
\end{Lem}
\begin{proof}
	It is obvious that $\tilde{u}\in PSH^-(\Omega)$.
	
	Since $X$ is compact, for every $j\in\Z^+$, we can divide $X$ into a finite pairwise disjoint collection of sets
	 of diameter less than $\dfrac{1}{2^j}$. We denote these sets by $U_{j,1},..., U_{j,m_j}$. 
	 We can furthermore assume that for every $1\leq k\leq m_{j+1}$, there exists $1\leq l\leq m_j$ such that
	 $U_{j+1, k}\subset U_{j, l}$.
	 
	 For every $j\in Z^+$, we define
	 \begin{center}
	 	$u_j(z)=\sum\limits_{k=1}^{m_j}\mu (U_{j, k})\sup\limits_{a\in U_{j,k}}u(z, a)\qquad$ and $\qquad\tilde{u}_j=(u_j)^*$.
	 \end{center}
 Then $\tilde{u}_j\in\F (\Omega)$. Moreover, by \cite{Ceg04}, we have
 \begin{center}
 	$\int\limits_{\Omega}(dd^c\tilde{u}_j)^n\leq M,$
 \end{center}
for all  $j\in Z^+$. 

By the semicontinuity of $u(z, .)$, we get that $\tilde{u}_j$ is decreasing to $\tilde{u}$ as $j\rightarrow\infty$. Hence,
$\tilde{u}\in\F(\Omega)$ and $\int\limits_{\Omega}(dd^c\tilde{u})^n\leq M$.
\end{proof}
Recall that if $u$ is a radial plurisubharmonic function then $u(z)=\chi (\log |z|)$ for some convex, increasing function $\chi$. We have the
following lemma:
\begin{Lem}\label{lem 2 main2}
	Let $u=\chi (\log |z|)$ be a radial plurisubharmonic function in $\B^{2n}$. Then, $u\in\F (\B^{2n})$ iff the following conditions hold
	\begin{itemize}
		\item [(i)] $\lim\limits_{t\to 0^-}\chi (t)=0$;
		\item [(ii)] $\lim\limits_{t\to 0^-}\dfrac{\chi (t)}{t}<\infty$.
	\end{itemize}
\end{Lem}
\begin{proof}
	It is clear that $(i)$ a necessary condition for  $u\in\F (\B^{2n})$. We need to show that, when $(i)$ is satisfied,
	the condition  $u\in\F (\B^{2n})$ is equivalent to $(ii)$.
	
	 If $(ii)$ is satisfied then there exists $k_0\gg 1$ such that $k_0t<\chi(t)$. Hence $u(z)>k_0\log|z|\in \F (\B^{2n})$. Thus,
	 $u\in\F (\B^{2n})$.
	 
	 Conversely, if $(ii)$ is not satisfied, we consider the functions $u_k=\max\{u, k\log |z| \}$.
	  Then, for every $k$, $u_k>u$ near $\partial\B^{2n}$.
	 Hence
	 \begin{center}
	 $\int\limits_{\Omega}(dd^cu)^n\geq\int\limits_{\Omega}(dd^cu_k)^n
	 =k^n\int\limits_{\Omega}(dd^c\log|z|)^n\stackrel{k\to\infty}{\longrightarrow}\infty.$
	  \end{center}
  Thus $u\notin\F(\B^{2n})$.
  
  The proof is completed.
\end{proof}
\begin{proof}[Proof of Theorem \ref{main 2}.]
Denote by $\mu$ the unique invariant probability measure on the unitary group $U(n)$.
 For every $z\in\B^{2n}$, we define
\begin{center}
	$\tilde{u}(z)=\int\limits_{U(n)}u(\phi (z))d\mu(\phi)=\dfrac{1}{c_{2n-1}|z|^{2n-1}}\int\limits_{\{|w|=|z|\}}u(w)d\sigma(w)$,
\end{center}
where $c_{2n-1}$ is the $(2n-1)$-dimensional volume of $\partial\B^{2n}$.

By Lemma \ref{lem 1 main2}, we have $\tilde{u}\in\F (\B^{2n})$. Since $\tilde{u}$ is radial, we have, by Lemma \ref{lem 2 main2},
\begin{center}
	$\lim\limits_{|z|\to 1^-}\dfrac{\tilde{u}(z)}{|z|-1}=\lim\limits_{|z|\to 1^-}\dfrac{\tilde{u}(z)}{\log|z|}<\infty$.
\end{center}
Hence
	\begin{center}
	$\lim\limits_{r\to 1^-}\dfrac{\int_{\{|z|=r\}}|u(z)|d\sigma(z)}{1-r}=M<\infty$.
\end{center}
Consequently, we have
\begin{center}
	$\limsup\limits_{d\to 0^+}\dfrac{Vol_{2n}(\{z\in\mathbb{B}^{2n} : \|z\|=1-d, u(z)<-Ad\})}{d}\leq\dfrac{M}{A}$,
\end{center}
for all $A>0$.

By using spherical coordinates to estimate integrals, we get the last assertion of Theorem \ref{main 2}.

The proof is completed.
\end{proof}
\section{Proof of Theorem \ref{main 3}}
\subsection{An approximation lemma}
In order to prove  Theorem \ref{main 3}, we need the following lemma:
\begin{Lem}\label{lem appr}
	Let $\Omega$ be a hyperconvex domain in $\C^n$ and $u\in PSH^-(\Omega)$. Assume that there are
	$u_j\in\F (\Omega)$, $j\in\N$, such that $u_j$ converges almost everywhere to
	$u$ as $j\rightarrow\infty$. If
	$\sup_{j>0}\int_{\Omega}(dd^cu_j)^n<\infty$ then $u\in\F(\Omega)$.
\end{Lem}
This lemma has been proved in \cite{NP09}. For the reader's convenience, we also give the details of
the proof. First, we need the following lemmas:
\begin{Lem}\label{the cegrell appro}\cite{Ceg04}
	Let $u\in PSH^-(\Omega)$. Then there exists a decreasing sequence of functions
	$u_j\in\mathcal{E}_0(\Omega)\cap C(\Omega)$ such that $\lim_{j\to \infty}u_j(z)=u(z)$
	for every $z\in\Omega$.
\end{Lem}
\begin{Lem}\label{compa NP}
	Let $u, v\in \mathcal{F}(\Omega)$ be such that $u\leq v$ on $\Omega$. Then
	$$\int_{\Omega}(dd^cu)^n\geq\int_{\Omega}(dd^cv)^n.$$
\end{Lem}
\begin{proof}
	Let $\{u_j\}_{j\in\mathbb{N}},\{v_j\}_{j\in\mathbb{N}}\subset\mathcal{E}_0(\Omega)$ be decreasing sequences such that $u_j\searrow u$, $v_j\searrow v$ on $\Omega$ and
	$$\sup\limits_{j>0}\int_{\Omega}(dd^cu_j)^n<+\infty,\ \ \sup\limits_{j>0}\int_{\Omega}(dd^cv_j)^n<+\infty.$$
	Replacing $v_j$ by $(1-\dfrac{1}{2^j})\max\{v_j,u_j\}$, we can assume that $v_j\geq u_j$. 
	By the Bedford-Taylor comparison principle \cite{BT76, BT82} (see also \cite{Kli91}), we obtain, for every $j$,
	$$\int_{\Omega}(dd^cu_j)^n\geq\int_{\Omega}(dd^cv_j)^n.$$
	Letting $j\rightarrow+\infty$, we get
	$$\int_{\Omega}(dd^cu)^n\geq\int_{\Omega}(dd^cv)^n,$$
	as desired.
\end{proof}
\begin{proof}[Proof of Lemma \ref{lem appr}]
	For every $k\geq 1$, we denote
	\begin{center}
		$u^k(z)=\sup\limits_{j\geq k}\max \{u, u_j\}$.
	\end{center}
	Then, we have
	\begin{itemize}
		\item [(i)]$v_k:=(u^k)^*\in PSH^-(\Omega)$ for all $k\geq 1$.
		\item [(ii)] $v_k$ is a decreasing sequence satisfying $v_k\geq u$ for every $k\geq 1$.
		\item [(iii)] $v_k=u^k$ almost everywhere and $u^k$ converges to $u$ almost everywhere.
	\end{itemize}
	By (iii),  we have $\lim_{k\to\infty}v_k=u$ almost everywhere. Since 
	$u$ and $\lim_{k\to\infty}v_k$ are plurisubharmonic, we get $u=\lim_{k\to\infty}v_k$.
	
	Since $0\geq v_k\geq u_k$, we have $v_k\in\F (\Omega)$. Moreover, by using Lemma
	\ref{compa NP}, we obtain
	\begin{center}
		$C:=\sup\limits_{j>0}\int\limits_{\Omega}(dd^cu_j)^n\geq \int\limits_{\Omega}(dd^cv_k)^n,$
	\end{center}
	for every $k\geq 1$.
	
	Now, it follows from Lemma \ref{the cegrell appro} that there exists a decreasing sequence
	$w_k\in\E_0 (\Omega)\cap C(\Omega)$ such that $\lim_{j\to\infty}w_j(z)=u(z)$ in $\Omega$.
	Replacing $w_j$ by $(1-j^{-1})w_j$, we can assume that $w_j(z)>u(z)$
	for every $j>0, z\in\Omega$. Applying Lemma \ref{compa NP}, we have
	\begin{center}
		$\int\limits_{\{v_k<w_j\}}(dd^cw_j)^n\leq \int\limits_{\{v_k<w_j\}}(dd^cv_k)^n\leq C,$
	\end{center}
	for every $j, k>0$.
	
	Letting $k\rightarrow \infty$, we get,
	\begin{center}
		$\int\limits_{\Omega}(dd^cw_j)^n\leq C,$
	\end{center}
	for every $j>0$.
	
	Thus, $u\in\F (\Omega)$.
\end{proof}
\subsection{Proof of Theorem \ref{main 3}}
For every $0<a<1$, we denote $S_a=\{\phi\in U(n): \|\phi-Id\|<a \}$ .

 For every $0<\epsilon, a<1$ and $z\in\B^{2n}_{1-\epsilon}:=\{w\in\C^n: \|w\|<1-\epsilon \}$, we define
 \begin{center}
 	$u_{a,\epsilon}(z)=(\sup\{u((1+r)\phi(z)): \phi\in S_a, 0\leq r\leq\epsilon \})^*$.
 \end{center}
 Then $u_{a, \epsilon}$ is plurisubharmonic in $\B^{2n}_{1-\epsilon}$ satisfying
 \begin{equation}\label{eq1 proof3}
 \lim\limits_{a\to 0^+}\lim\limits_{\epsilon\to 0^+}u_{a, \epsilon}(z)=u(z),
 \end{equation}
 for every $z\in\Omega$.
 
  Moreover, for $z\neq 0$,
 \begin{equation}\label{eq2 proof3}
 	u_{a,\epsilon}(z)=(\sup\{u(\xi): \xi\in B_{a,\epsilon, z} \})^*,
 \end{equation}
 where
 \begin{center}
 $B_{a,\epsilon, z}=\{\xi\in\C^n: \|\dfrac{z}{\|z\|}-\dfrac{\xi}{\|\xi\|}\|<a, 
 \|z\|\leq\|\xi\|\leq(1+\epsilon)\|z\|  \}.$
 \end{center}
 It is obvious that there exist $C_1, C_2>0$ such that
 \begin{equation}\label{eq3 proof3}
 	C_1 a^{2n-1}\epsilon<Vol_{2n}(B_{a,\epsilon, z})<C_2a^{2n-1}\epsilon, 
 \end{equation}
 for every $0<\epsilon, a<1/2$ and $1/2<\|z\|<1-a$.
 
By \eqref{eq main3}, \eqref{eq2 proof3} and \eqref{eq3 proof3}, for every $1/2>a>0$, there exists
$\epsilon_a>0$ such that, for every $\epsilon_a>3\epsilon\geq 1-\|z\|\geq\epsilon>0$,
we have
\begin{equation}
u_{a, \epsilon}(z)\geq -3A\epsilon.
\end{equation}
For each $1/2>a>0$ and $\epsilon_a>3\epsilon>0$, we consider the following function
\begin{center}
	$\tilde{u}_{a, \epsilon}(z)=
	\begin{cases}
3A(-1+|z|^2)\quad\mbox{ if }\quad 1-\epsilon\leq \|z\|\leq 1,\\
	\max\{3A(-1+|z|^2), u_{a, \epsilon}(z)-6A\epsilon \}\quad\mbox{ if }
	\quad 1-3\epsilon\leq \|z\|\leq 1-\epsilon,\\
	u_{a, \epsilon}(z)-6A\epsilon\quad\mbox{ if }\quad \|z\|\leq 1-3\epsilon.
	\end{cases}$
\end{center}
Then $\tilde{u}_{a, \epsilon}\in\F(\B^{2n})$ and
\begin{center}
		$\int\limits_{\B^{2n}}(dd^c\tilde{u}_{a, \epsilon})^n
		=\int\limits_{\B^{2n}}(dd^c3A(-1+|z|^2))^n<\infty$,
\end{center}
for every $1/2>a>0$ and $\epsilon_a>3\epsilon>0$. 

Moreover, $\tilde{u}_{a, \epsilon}\stackrel{a. e.}{\longrightarrow}u$ as $a, \epsilon\searrow 0$.
Hence, by Lemma \ref{lem appr}, we have $u\in\F (\Omega)$.

The proof is completed.


\begin{thebibliography}{BB}
\bibitem[Aha07]{Aha07}{P. AHAG: \it A Dirichlet problem for the complex Monge-Amp\`ere operator in $\F(f)$. \rm  Michigan Math. J. \bfseries{55} \rm (2007), no. 1, 123--138. }
\bibitem[ACCP09]{ACCP09}{P. AHAG, U. CEGRELL, R. CZYZ, H.-H. PHAM: \it Monge-Amp\`ere measures on pluripolar sets. \rm  J. Math. Pures Appl. (9) \bfseries{92} \rm  (2009), no. 6, 613--627.}
\bibitem[Blo06]{Blo06}{Z. BLOCKI: \it The domain of definition of the complex Monge-Amp\`ere operator. \rm Amer. J. Math.	\bf{128} \rm (2006), no.2, 519--530.}
\bibitem[BT76]{BT76}{E. BEDFORD, B. A. TAYLOR: \it The Dirichlet problem for a complex Monge-Amp\`ere equation. \rm Invent. Math. \bfseries{37} \rm (1976), no. 1,  1--44 .}
\bibitem[BT82]{BT82}{E. BEDFORD, B. A. TAYLOR: \it A new capacity for plurisubharmonic functions.	\rm Acta Math. \bfseries{149} \rm (1982), no. 1-2, 1--40.}
\bibitem [Ceg98]{Ceg98} {U. CEGRELL: \it Pluricomplex energy. \rm Acta Math. {\bf 180} \rm (1998), no. 2, 187--217.}
\bibitem[Ceg04]{Ceg04}{U. CEGRELL: \it The general definition of the complex Monge-Amp\`ere operator. \rm (English, French summary) Ann. Inst. Fourier (Grenoble) \bfseries{54} \rm (2004), no. 1, 159--179.}
\bibitem[Kli91]{Kli91}{M. KLIMEK: \it Pluripotential theory, \rm Oxford Univ. Press, Oxford, 1991.}
\bibitem[NP09]{NP09}{V.K. NGUYEN, H.-H. PHAM: \it A comparison principle for the complex Monge-Amp\`ere operator in Cegrell's classes and applications. \rm Trans. Amer. Math. Soc. \bfseries{361} \rm (2009), no. 10, 5539--5554.}
\end{thebibliography}
\end{document}